\documentclass[12pt, reqno, a4paper]{amsart}
\usepackage{times}

\usepackage[utf8]{inputenc}
\usepackage[USenglish]{babel}
\usepackage{amsmath,amsthm,amssymb,amsfonts}
\usepackage{mathtools}
\usepackage{mathrsfs}
\usepackage{bbm}
\usepackage{booktabs}

\usepackage{indentfirst}
\usepackage{enumitem}
\usepackage{xcolor}

\usepackage{float}

\usepackage[breaklinks=true, bookmarksopenlevel=1, bookmarksdepth=2]{hyperref}
\hypersetup{
colorlinks,
   linkcolor={cyan!80!black},
   citecolor={cyan!80!black},
 urlcolor={cyan!80!black}
}

\allowdisplaybreaks

\newtheorem{thm}{Theorem}[section]

\newtheorem{lem}[thm]{Lemma}
\newtheorem{prop}[thm]{Proposition}

\theoremstyle{plain} % just in case the style had changed
\newcommand{\thistheoremname}{}
\newtheorem*{genericthm}{\thistheoremname}

\theoremstyle{definition}

\theoremstyle{remark}
\newtheorem{rem}[thm]{Remark}

\numberwithin{equation}{section}

\newcommand{\N}{\mathbb{N}}      % N = Naturals
\newcommand{\eps}{\varepsilon}   % epsilon

\usepackage[margin=1in]{geometry}

\restylefloat{table}

\setlist[itemize]{leftmargin=*}
\setlist[enumerate]{leftmargin=*}

\begin{document}

\title{On Vu's restricted box estimate in Waring's problem}%

\author{Christian T\'afula}%
\address{Institute of Mathematics, Statistics\\
and Computer Science\\
University of S\~ao Paulo\\
Rua do Mat\~ao, 1010\\
S\~ao Paulo, SP 05508-090\\
Brazil}
\curraddr{}
\email{tafula@ime.usp.br}
\thanks{}


% ----------------------------------------------------------------
\begin{abstract}
 In 2000, Vu proved that the number of solutions of $x_1^k + \cdots + x_s^k = N$ in an arbitrary box satisfies the expected Hardy--Littlewood upper bound with a power-saving error term, for $s \geq O(8^k k^3)$. We show that one may take $s\geq k^2 - k + O(\sqrt{k})$.
\end{abstract}

\maketitle
% ----------------------------------------------------------------

\section{Introduction}
 Let $k \geq 2$. In his work on thin subbases for Waring's problem, Vu \cite{vvu00wp} proved that, for $s$ sufficiently large in terms of $k$, there exists $A \subseteq \{n^k : n\in\N\}$ such that the number of representations of $N$ as a sum of $s$ elements of $A$ is $\asymp \log N$. The main number-theoretic input in Vu's argument is a restricted version of Waring's problem, estimating the number of solutions to the equation $x_1^k+\cdots+x_s^k=N$ inside an arbitrary box.

 For positive integers $P_1,\ldots,P_s$, write $\mathrm{Root}(P_1,\ldots,P_s;N)$ for the number of solutions of $x_1^k+\cdots+x_s^k=N$ with $1 \leq x_j \leq P_j$. Vu proved that there exists $s_0(k)$ such that, whenever $s \geq s_0(k)$, there is some $\delta=\delta(k,s)>0$ for which
 \begin{equation}
  \mathrm{Root}(P_1,\ldots,P_s;N) \ll N^{-1}\prod_{j=1}^s P_j+\bigg(\prod_{j=1}^s P_j\bigg)^{1-k/s-\delta} \label{main-est}
 \end{equation}
 uniformly in $N$ and in the side lengths $P_1,\ldots,P_s$. Vu's argument gives $s_0(k) \ll 8^k k^3$.\footnote{We note that the statement in \cite{vvu00wp} gives $8^k k^2$, but the computation in the proof gives the additional factor of $k$.}

 Vu's thin subbasis theorem was later sharpened by Wooley \cite{woo03}, who bypassed this restricted box estimate and obtained a much stronger threshold for the thin basis problem itself. More recently, Pliego \cite{pli24} obtained further refinements. Thus, from the point of view of thin Waring subbases, Vu's original lemma is no longer the sharpest available route. The purpose of this note is different: we return to Vu's restricted box estimate itself and improve its threshold directly.

 We shall use recent progress on Vinogradov's mean value theorem. Let
 \[ J_{s,k}(X) := \int_{[0,1)^k}\bigg|\sum_{1 \leq x \leq X} e(\alpha_1 x + \alpha_2 x^2 + \cdots + \alpha_k x^k)\bigg|^{2s}\, \mathrm{d}\boldsymbol{\alpha}. \]
 The main conjecture in Vinogradov's mean value theorem asserts that, for every $\eps>0$,
 \[ J_{s,k}(X) \ll_{\eps,s,k} X^\eps (X^s+X^{2s-k(k+1)/2}). \]
 The cubic case was obtained by Wooley via efficient congruencing \cite{woo12}, and the conjecture was subsequently proved in general by Bourgain, Demeter and Guth \cite{boudemgut16} and by Wooley through nested efficient congruencing \cite{woo19}.

 The particular input we need is Hua's lemma. The classical form of the lemma gives
 \begin{equation}
  \int_0^1 \bigg|\sum_{x \leq X} e(\alpha x^k)\bigg|^s \mathrm{d}\alpha \ll_\eps X^{s-k+\eps} \label{hua-est}
 \end{equation}
 for every integer $s \geq 2^k$. For large $k$, we use the sharper estimates of Wooley \cite{woo19}. Refining work of Bourgain \cite{bou17}, Wooley proved that the same estimate holds for every real number
 \[ s \geq k^2-k+2\lfloor \sqrt{2k+2}\rfloor-\theta(k), \]
 where
 \[ \theta(k) := \begin{cases}
                  1, & 2k+2 \leq \lfloor \sqrt{2k+2}\rfloor^2+\lfloor \sqrt{2k+2}\rfloor,\\
                  2, & 2k+2 > \lfloor \sqrt{2k+2}\rfloor^2+\lfloor \sqrt{2k+2}\rfloor.
                 \end{cases} \]
 This gives an explicit version of the bound $k^2-k+O(\sqrt{k})$. Comparing this with Hua's classical exponent $2^k$, one sees that Hua's lemma is better for $k=2$ and $k=4$, the two estimates coincide for $k=3$, and Wooley's estimate is better for every $k \geq 5$.

 Our main result is the following:

 \begin{thm}\label{main-thm}
  Let $k \geq 2$, and let $s$ be an integer satisfying
  \[ s \geq \begin{cases}
          2^k + 1, & 2 \leq k \leq 4,\\
          k^2-k+2\lfloor \sqrt{2k+2}\rfloor-\theta(k) + 1, & k \geq 5.
         \end{cases} \]
  Then \eqref{main-est} holds for some $\delta=\delta(k,s)>0$.
 \end{thm}

 The proof follows the broad circle method structure of Vu's original argument, but with a different treatment of the minor arcs. Put
 \[ f_j(\alpha):=\sum_{x \leq P_j}e(\alpha x^k). \]
 By orthogonality,
 \begin{equation}
  \mathrm{Root}(P_1,\ldots,P_s;N)=\int_0^1 f_1(\alpha)\cdots f_s(\alpha)\,e(-N\alpha)\,\mathrm{d}\alpha. \label{root-int}
 \end{equation}
 Vu's minor arc estimate applies pointwise Weyl bounds to a suitable collection of the sums $f_j$, and this is the step responsible for the exponential dependence on $k$. Instead, we reserve one variable for a pointwise minor arc saving and control the remaining $s-1$ variables by a Hua-type mean value estimate. This gives the threshold in Theorem \ref{main-thm}.

 The major arcs require a separate argument in the balanced case, where the usual major arc approximation gives the expected contribution $N^{-1}\prod_j P_j$, up to an acceptable power-saving error. The complementary unbalanced case is handled by a weighted consequence of the same Hua-type mean value estimate: after fixing the smallest variable, the remaining $s-1$ variables already give a power saving. Combining these estimates proves Theorem \ref{main-thm}.

\section{Proof of Theorem \ref{main-thm}}

\subsection{Setup}
 Set
 \[ H_k := \begin{cases}
             2^k, & 2 \leq k \leq 4,\\
             k^2-k+2\lfloor \sqrt{2k+2}\rfloor-\theta(k), & k \geq 5,
            \end{cases} \]
 so that \eqref{hua-est} holds for every integer $s\geq H_k$, and Theorem \ref{main-thm} assumes $s\geq H_k+1$. Throughout the proof, we shall use
 \[ K:=2^{k-1}, \qquad \lambda:=\frac{1}{12k}, \qquad \tau:=\frac{1}{12ks^2}. \]
 Thus
 \[ k\lambda<\frac{1}{6},\qquad \lambda+\frac{1}{3}<1,\qquad \frac{1}{6}-\lambda>0,\qquad \tau=\frac{\lambda}{s^2}. \]

 After reordering the side lengths, write
 \[ P_1\leq\cdots\leq P_s,\qquad P:=\prod_{j=1}^sP_j,\qquad X:=P_s, \]
 and use $f_j$ as in \eqref{root-int}. We may harmlessly assume that
 \[ (N/s)^{1/k}\leq X\leq N^{1/k}. \]
 Indeed, if $X<(N/s)^{1/k}$ then there are no solutions, while if $X>N^{1/k}$ we replace each $P_j$ by $\min(P_j,N^{1/k})$ and reorder. This does not change the set of solutions and only decreases $P$, so an estimate for the truncated box implies the same estimate for the original one.

%%%%%%%%%%%%%%%%%%%%%%%%%%%
\subsection{Balanced case}
 We first treat boxes whose side lengths are close to the largest one.

 \begin{prop}\label{balanced-box}
  Let $s\geq H_k+1$ be an integer, and suppose that $P_1\geq X^{1-\lambda}$. Then there exists $\delta_0=\delta_0(k,s)>0$ such that
  \[ \mathrm{Root}(P_1,\ldots,P_s;N) \ll N^{-1}P+P^{1-k/s-\delta_0}. \]
 \end{prop}

 For the results in this subsection, we assume the hypotheses of Proposition \ref{balanced-box}. Define the major arcs by
 \[ \mathfrak{M} := \bigcup_{1\leq q\leq X^{1/6}} \bigcup_{\substack{1\leq a\leq q\\ (a,q)=1}} \mathfrak{M}(q,a),\qquad \mathfrak{M}(q,a) := \bigg\{\alpha\in[0,1]: \bigg\|\alpha-\frac{a}{q}\bigg\|\leq X^{1/6-k}\bigg\}, \]
 where $\|\cdot\|$ denotes distance to the nearest integer, and put $\mathfrak{m} :=[0,1]\setminus \mathfrak{M}$. For large $X$, these arcs are disjoint, since distinct fractions $a/q$ and $a'/q'$ with $q,q'\leq X^{1/6}$ are separated by at least $X^{-1/3}$, while $2X^{1/6-k}<X^{-1/3}$.

 \begin{lem}[Minor arcs]\label{balanced-minor}
  We have
  \[ \int_{\mathfrak{m}}\prod_{j=1}^s |f_j(\alpha)|\,\mathrm{d}\alpha \ll P^{1-k/s-\frac{1}{24Ks^2}}. \]
 \end{lem}
 \begin{proof}
  By Dirichlet's approximation theorem \cite[Lemma 2.1]{vaughan97}, for every $\alpha\in[0,1]$ there are coprime integers $0\leq a\leq q$ with $1\leq q\leq X^{k-1/6}$ and
  \[ \bigg|\alpha-\frac{a}{q}\bigg|\leq \frac{X^{1/6-k}}{q}. \]
  If $\alpha\in\mathfrak{m}$, then necessarily $q>X^{1/6}$. Weyl's inequality \cite[Lemma 2.4]{vaughan97} then gives
  \[ |f_s(\alpha)|\ll_\eps X^{1+\eps}(q^{-1}+X^{-1}+qX^{-k})^{1/K}. \]
  Since $X^{1/6}<q\leq X^{k-1/6}$, taking $\eps = 1/12K$ gives
  \[ \sup_{\alpha\in\mathfrak{m}}|f_s(\alpha)|\ll X^{1-\frac{1}{12K}}. \]
  
  As $s-1\geq H_k$, \eqref{hua-est} gives 
  \[ \int_0^1 |f_j(\alpha)|^{s-1}\,\mathrm{d}\alpha \ll_\eta P_j^{s-1-k+\eta} \]
  for every $\eta>0$. Thus $\int_0^1 |f_j(\alpha)|^s\,\mathrm{d}\alpha \ll_\eta P_j^{s-k+\eta}$ for all $j$, while for $j=s$ we also have
  \[ \int_{\mathfrak{m}} |f_s(\alpha)|^s\,\mathrm{d}\alpha \leq \Big(\sup_{\alpha\in\mathfrak{m}}|f_s(\alpha)|\Big) \int_0^1 |f_s(\alpha)|^{s-1}\,\mathrm{d}\alpha \ll_\eta X^{s-k-\frac{1}{12K}+\eta}. \]
  By H\"older's inequality,
  \begin{align*}
   \int_{\mathfrak{m}}\prod_{j=1}^s |f_j(\alpha)|\,\mathrm{d}\alpha &\leq \bigg(\int_{\mathfrak{m}}|f_s(\alpha)|^s\,\mathrm{d}\alpha\bigg)^{1/s} \prod_{j=1}^{s-1}\bigg(\int_0^1 |f_j(\alpha)|^s\,\mathrm{d}\alpha\bigg)^{1/s} \\
   &\ll_{\eta} X^{1-k/s-\frac{1}{12Ks} + \eta/s}\, (P/X)^{1-k/s+\eta/s} \\
   &= X^{-\frac{1}{12Ks}}P^{1-k/s + \eta/s}.
  \end{align*}
  Since $X = P_s\geq P^{1/s}$, choosing $\eta=1/24Ks$ yields
  \[ \int_{\mathfrak{m}}\prod_{j=1}^s |f_j(\alpha)|\,\mathrm{d}\alpha \ll P^{1-k/s-\frac{1}{24Ks^2}}. \qedhere \]
 \end{proof}

 Now we deal with the major arcs. For $Y\leq X$, define
 \[ f_Y(\alpha):=\sum_{x\leq Y}e(\alpha x^k), \qquad v_Y(\beta):=\int_0^Y e(\beta \gamma^k)\,\mathrm{d}\gamma, \qquad S(q,a):=\sum_{r=1}^q e\bigg(\frac{ar^k}{q}\bigg). \]

 \begin{lem}\label{major-approx}
  If $q\leq X^{1/6}$, $(a,q)=1$, and $|\beta|\leq X^{1/6-k}$, then
  \[ f_Y(a/q+\beta)=q^{-1}S(q,a)v_Y(\beta)+O(X^{1/3}) \]
  uniformly for $Y\leq X$. 
 \end{lem}
 \begin{proof}
  Writing the sum in residue classes modulo $q$ gives
  \begin{equation}
   f_Y(a/q+\beta)=\sum_{r=1}^q e\bigg(\frac{ar^k}{q}\bigg)\sum_{\substack{x\leq Y\\ x\equiv r(q)}}e(\beta x^k). \label{aqbeta}
  \end{equation}
  Put $F(t) := e(\beta t^k)$, and let
  \[ A_r(u) := \sum_{\substack{x\leq u\\ x\equiv r(q)}}1-\frac{u}{q}. \]
  By partial summation,
  \[ \sum_{\substack{x\leq Y\\ x\equiv r(q)}}F(x) = \int_{0}^Y F(u)\,\mathrm{d}\bigg(\frac{u}{q}+A_r(u)\bigg) = \frac{1}{q}\int_0^Y F(u)\,\mathrm{d}u+\int_{0}^Y F(u)\,\mathrm{d}A_r(u). \]
  Integrating the last term by parts, and using that $A_r(u)=O(1)$ uniformly in $u$ and $r$, we get
  \[ \int_{0}^Y F(u)\,\mathrm{d}A_r(u) \ll |F(0)|+|F(Y)|+\int_0^Y |F'(u)|\,\mathrm{d}u. \]
  Thus, since $|F(u)|=1$, we have
  \begin{equation*}
   \sum_{\substack{x\leq Y\\ x\equiv r(q)}}F(x)=\frac{1}{q}\int_0^Y F(t)\,\mathrm{d}t+O\bigg(1+\int_0^Y |F'(t)|\,\mathrm{d}t\bigg)
  \end{equation*}
  uniformly in $1\leq r\leq q$.

  Now $|F'(t)|=2\pi k|\beta|t^{k-1}$, and hence $\int_0^Y |F'(t)|\,\mathrm{d}t\ll_k |\beta|Y^k\leq X^{1/6}$. It follows that
  \[ \sum_{\substack{x\leq Y\\ x\equiv r(q)}} e(\beta x^k) = \frac{1}{q} v_Y(\beta)+O(X^{1/6}). \]
  Substituting this into \eqref{aqbeta} yields, since $q\leq X^{1/6}$,
  \begin{align*}
   f_Y(a/q+\beta) &= \frac{1}{q} v_Y(\beta)\sum_{r=1}^q e\bigg(\frac{ar^k}{q}\bigg) + O(qX^{1/6}) \\
   &= q^{-1} S(q,a) v_Y(\beta) + O(X^{1/3}). \qedhere
  \end{align*}
 \end{proof}
 
 \begin{lem}\label{v-bound}
  We have
  \[ |v_Y(\beta)|\ll \min(Y,|\beta|^{-1/k}). \]
 \end{lem}
 \begin{proof}
  The bound $|v_Y(\beta)|\leq Y$ is trivial. If $\beta\neq 0$, then the change of variables $u=|\beta|^{1/k}\gamma$ gives
  \[ v_Y(\beta)=|\beta|^{-1/k}\int_0^{|\beta|^{1/k}Y} e(\operatorname{sgn}(\beta)u^k)\,\mathrm{d}u. \]
  Writing $T:=|\beta|^{1/k}Y$ and $\sigma:=\operatorname{sgn}(\beta)$, if $T\geq 1$ then
  \begin{align*}
   \int_0^T e(\sigma u^k)\,\mathrm{d}u &\ll 1+\bigg|\int_1^T e(\sigma u^k)\,\mathrm{d}u\bigg| \\
   &= 1+\bigg|\left[\frac{e(\sigma u^k)}{2\pi i\sigma k u^{k-1}}\right]_1^T + \frac{k-1}{2\pi i\sigma k}\int_1^T u^{-k}e(\sigma u^k)\,\mathrm{d}u\bigg| \\
   &\ll 1.
  \end{align*}
  The case $T<1$ is trivial. Hence $|v_Y(\beta)|\ll |\beta|^{-1/k}$.
 \end{proof}

 \begin{lem}[Singular integral]\label{sing-int}
  We have
  \[ \int_{-\infty}^{\infty}\prod_{j=1}^s v_{P_j}(\beta)\, e(-N\beta)\,\mathrm{d}\beta \ll PX^{-k}. \]
 \end{lem}
 \begin{proof}
  Since $s>k$, the bound $|v_{P_j}(\beta)|\ll \min(P_j,|\beta|^{-1/k})$ from Lemma \ref{v-bound} shows that the integral is absolutely convergent. Now make the change of variables $u_j=\gamma_j^k$ in the definition of $v_{P_j}$, so that
  \[ v_{P_j}(\beta)=\frac1k\int_0^{P_j^k} u^{1/k-1}e(\beta u)\,\mathrm{d}u. \]
  By Fourier inversion,
  \[ J:= \int_{-\infty}^{\infty}\prod_{j=1}^s v_{P_j}(\beta)\, e(-N\beta)\,\mathrm{d}\beta = \frac{1}{k^s}\int_{\substack{0\leq u_j\leq P_j^k\\ u_1+\cdots+u_s=N}} (u_1\cdots u_s)^{1/k-1}\,\mathrm{d}\sigma. \]
 
  On the hyperplane $u_1+\cdots+u_s=N$, at least one coordinate is $\geq N/s$. We cover the hyperplane by the sets
  \[ E_i:=\{u_1+\cdots+u_s=N,\ 0\leq u_j\leq P_j^k,\ u_i\geq N/s\} \qquad (1\leq i\leq s). \]
  It is enough to bound the contribution from each $E_i$. On $E_i$ we have $u_i\geq N/s\gg_s X^k$, so since $1/k-1<0$, it holds that $u_i^{1/k-1}\ll_s X^{1-k}$.
  Solving for $u_i=N-\sum_{j\neq i}u_j$, and bounding the remaining domain by the full box $0\leq u_j\leq P_j^k$ for $j\neq i$, the contribution from $E_i$ is at most
  \[ X^{1-k}\prod_{j\neq i}\int_0^{P_j^k}u_j^{1/k-1}\,\mathrm{d}u_j \ll X^{1-k}\prod_{j\neq i}P_j. \]
  Moreover, $E_i$ is empty unless $P_i^k\geq N/s\gg_s X^k$, so whenever it contributes we have
  \[ X^{1-k}\prod_{j\neq i}P_j=X^{1-k}\frac{P}{P_i}\ll_s PX^{-k}. \]
  Summing over $1\leq i\leq s$ gives $J\ll PX^{-k}$.
 \end{proof}

 \begin{lem}[Major arcs]\label{balanced-major}
  There exists $\delta'=\delta'(k,s)>0$ such that
  \[ \int_{\mathfrak M} f_1(\alpha)\cdots f_s(\alpha)\, e(-N\alpha)\,\mathrm{d}\alpha \ll N^{-1}P+P^{1-k/s-\delta'}. \]
 \end{lem}
 \begin{proof}
  For $\alpha = a/q + \beta\in\mathfrak{M}(q,a)$, write
  \[ V_j(\alpha):=q^{-1}S(q,a)v_{P_j}(\beta). \]
  By Lemma \ref{major-approx}, $f_j(\alpha)=V_j(\alpha)+O(X^{1/3})$. Also $|V_j(\alpha)|\leq P_j$, since $|S(q,a)|\leq q$ and $|v_{P_j}(\beta)|\leq P_j$. As $P_j\geq P_1\geq X^{1-\lambda}$, we have $X^{1/3}\leq P_jX^{-2/3+\lambda}$. Thus, expanding the product one error at a time gives
  \[ \prod_{j=1}^s f_j(\alpha)=\prod_{j=1}^s V_j(\alpha)+O(PX^{-2/3+\lambda}). \]
  Since $\int_{\mathfrak M}\mathrm{d}\alpha \ll X^{1/2-k}$, the total contribution of this error is $O(PX^{-k-1/6+\lambda})$. Using $P\leq X^s$, this is $O(P^{1-k/s-\delta_a})$, where one may take $\delta_a=(1/6-\lambda)/s>0$. Hence
  \begin{align*}
   \int_{\mathfrak M} f_1(\alpha)\cdots f_s(\alpha)\,e(-N\alpha)\,\mathrm{d}\alpha &= \int_{\mathfrak M} V_1(\alpha)\cdots V_s(\alpha)\,e(-N\alpha)\,\mathrm{d}\alpha + O(P^{1-k/s-\delta_a})\\
   &= \sum_{q\leq X^{1/6}}\sum_{\substack{1\leq a\leq q\\(a,q)=1}} q^{-s}S(q,a)^s e(-Na/q) \\
   &\hspace{4em}\times \int_{-X^{1/6-k}}^{X^{1/6-k}}\prod_{j=1}^s v_{P_j}(\beta)\,e(-N\beta)\,\mathrm{d}\beta +O(P^{1-k/s-\delta_a}).
  \end{align*}

  We now compare this expression with the completed singular series and singular integral. Since $|S(q,a)|\ll_k q^{1-1/k}$ (cf. \cite[Theorem 4.2]{vaughan97}) and $s>2k$,
  \[ \sum_{q=1}^{\infty}\sum_{\substack{1\leq a\leq q\\(a,q)=1}}\bigg|\frac{S(q,a)}{q}\bigg|^s \ll \sum_{q=1}^{\infty}q^{1-s/k}<\infty. \]
  Moreover, by Lemma \ref{sing-int}, the full singular integral is $O(PX^{-k})=O(N^{-1}P)$. Thus the completed contribution is $O(N^{-1}P)$, and it remains only to estimate the two tails.

  First, since $k\lambda<1/6$, for $|\beta|\geq X^{1/6-k}$ we have $|\beta|^{-1/k}\leq X^{1-1/6k}\leq X^{1-\lambda}\leq P_j$ for every $j$. Hence, by Lemma \ref{v-bound},
  \[ \int_{|\beta|>X^{1/6-k}}\prod_{j=1}^s |v_{P_j}(\beta)|\,\mathrm{d}\beta \ll \int_{X^{1/6-k}}^\infty \beta^{-s/k}\,\mathrm{d}\beta \ll X^{s-k-\frac{1}{6}(s/k-1)}. \]
  Since $P\geq X^{s(1-\lambda)}$, this is $O(P^{1-k/s-\delta_b})$, where one may take
  \[ \delta_b =\frac{\frac{1}{6}(s/k-1)-\lambda(s-k)}{s}=\frac{s-k}{12ks}. \]

  Second, we complete the singular series. Since the integral has now been replaced by the full singular integral, Lemma \ref{sing-int} gives a contribution
  \[ PX^{-k}\sum_{q>X^{1/6}}q^{1-s/k}\ll PX^{-k-\frac{1}{6}(s/k-2)}. \]
  Since $P\leq X^s$, this is $O(P^{1-k/s-\delta_c})$, where one may take
  \[ \delta_c =\frac{s/k-2}{6s}=\frac{s-2k}{6ks}. \]

  Combining the approximation error, the completed singular series and singular integral contribution, and the two tails, proves the result with $\delta':=\min\{\delta_a,\delta_b,\delta_c\}>0$.
 \end{proof}

 \begin{proof}[Proof of Proposition \ref{balanced-box}]
  By \eqref{root-int}, Lemmas \ref{balanced-minor} and \ref{balanced-major} give
  \begin{align*}
   \mathrm{Root}(P_1,\ldots,P_s;N) &\leq \int_{\mathfrak{m}}\prod_{j=1}^s |f_j(\alpha)|\,\mathrm{d}\alpha +\bigg|\int_{\mathfrak{M}} f_1(\alpha)\cdots f_s(\alpha)\,e(-N\alpha)\,\mathrm{d}\alpha\bigg| \\
   &\ll N^{-1}P+P^{1-k/s-\delta_0},
  \end{align*}
  where $\delta_0 = \min\{\frac{1}{24Ks^2},\delta'\}$.
 \end{proof}

%%%%%%%%%%%%%%%%%%%%%%%%%%%
\subsection{Unbalanced case}
 We now treat the complementary case, where the smallest side length is too small compared with the geometric mean of the box.

 \begin{lem}\label{weighted-tail}
  Let $s \geq H_k$. For every $\eps>0$ and every $m \geq 1$, one has
  \[ \sum_{x_1^k+\cdots+x_s^k=m} (x_1\cdots x_s)^{-1+k/s} \ll_{\eps} m^\eps. \]
 \end{lem}
 \begin{proof}
  We decompose each variable dyadically: for each $i$ we take $D_i=2^{j_i}$ and restrict $x_i$ to the interval $D_i\leq x_i < 2D_i$. Since any solution has $x_i\leq m^{1/k}$, there are $O((\log m)^s)$ relevant choices of $(D_1,\ldots,D_s)$.

  In one such dyadic box, the number of solutions is bounded by
  \[ \int_0^1 \prod_{i=1}^s \bigg|\sum_{D_i\leq x< 2D_i} e(\alpha x^k)\bigg|\, \mathrm{d}\alpha. \]
  By H\"older's inequality and \eqref{hua-est}, this is $\ll_\eps \prod_{i=1}^s D_i^{1-k/s+\eps/s}$. On the same box, the weight $(x_1\cdots x_s)^{-1+k/s}$ is $\ll \prod_{i=1}^s D_i^{-1+k/s}$. Thus the contribution of this box to the weighted sum is 
  \[ \ll_\eps \prod_{i=1}^s D_i^{\eps/s}. \]
  
  If the box contributes at all, then there is a solution with $D_i\leq x_i<2D_i$, so
  \[ D_1^k+\cdots+D_s^k\leq x_1^k+\cdots+x_s^k=m. \]
  By AM--GM,
  \[ (D_1\cdots D_s)^{1/s}\leq \bigg(\frac{D_1^k+\cdots+D_s^k}{s}\bigg)^{1/k}\ll m^{1/k}, \]
  therefore $\prod_i D_i^{\eps/s}\ll m^{\eps/k}$. Summing over the $O((\log m)^s)$ boxes gives the result.
 \end{proof}
 
 \begin{prop}\label{unbalanced-prop}
  Let $s\geq H_k+1$ be an integer, and suppose that $P_1\leq P^{1/s-\tau}$. Then there exists $\delta_1=\delta_1(k,s)>0$ such that
  \[ \mathrm{Root}(P_1,\ldots,P_s;N) \ll P^{1-k/s-\delta_1}. \]
 \end{prop}
 \begin{proof}
  Put $T:=P_2\cdots P_s$. For each fixed $x_1$, the remaining variables satisfy
  \[ x_2^k+\cdots+x_s^k=N-x_1^k. \]
  We apply Lemma \ref{weighted-tail} with $s-1$ in place of $s$. Every solution $x_2^k+\cdots + x_s^k = m$ with $1\leq x_i\leq P_i$ ($2\leq i\leq s$) contributes $(x_2\cdots x_s)^{-1+k/(s-1)} \geq T^{-1+k/(s-1)}$ to the weighted sum, because $-1+k/(s-1)<0$. Hence, for $m=N-x_1^k$,
  \[ \#\{x_2^k+\cdots+x_s^k=N-x_1^k : 1\leq x_i\leq P_i\ (2\leq i\leq s)\} \ll_\eps T^{1-k/(s-1)}N^\eps. \]
  As there are at most $P_1=P/T$ choices for $x_1$, we get
  \[ \mathrm{Root}(P_1,\ldots,P_s;N) \ll_\eps P T^{-k/(s-1)}N^\eps. \]
  If there are no solutions, there is nothing to prove. Otherwise, $N\leq sP_s^k\leq sP^k$, and hence $N^\eps\ll_s P^{k\eps}$. Thus, as $T = P/P_1 \geq P^{(s-1)/s+\tau}$,
  \[ \mathrm{Root}(P_1,\ldots,P_s;N) \ll_\eps P^{1-k/s-k\tau/(s-1)+k\eps}. \]
  Taking $\eps = \frac{\tau}{2(s-1)}$, we get
  \[ \mathrm{Root}(P_1,\ldots,P_s;N) \ll P^{1-k/s-\delta_1}, \]
  where we may take $\delta_1=\frac{k\tau}{2(s-1)}=\frac{1}{24s^2(s-1)}$.
 \end{proof}

 We finish by combining the balanced and unbalanced estimates.

 \begin{proof}[Proof of Theorem \ref{main-thm}]
  If $P_1\leq P^{1/s-\tau}$, Proposition \ref{unbalanced-prop} gives the required estimate. Otherwise $P_1>P^{1/s-\tau}$. Since $P\geq P_1^{s-1}P_s$, we have
  \[ P_s\leq P^{1/s+(s-1)\tau}. \]
  Since $\tau=\lambda/s^2$,
  \[ P^{1/s-\tau}\geq (P^{1/s+(s-1)\tau})^{1-\lambda}. \]
  Hence $P_1\geq P_s^{1-\lambda}=X^{1-\lambda}$, and Proposition \ref{balanced-box} applies.
 \end{proof}
 
 \begin{rem}
  We have stated the result with Wooley's explicit exponent because this gives a clean numerical improvement of Vu's threshold. The proof, however, shows more generally that any improvement in the available Hua-type mean value estimate for $\sum_{x \leq X}e(\alpha x^k)$ immediately improves the threshold in Vu's restricted box lemma.
 \end{rem}

%%%%%%%%%%%%%%%%%%%%%%
\addtocontents{toc}{\protect\setcounter{tocdepth}{0}}
\section*{Acknowledgements}
 The author acknowledges the support of the São Paulo Research Foundation (FAPESP), Brazil, Process No.~2025/15961-3.
 
\addtocontents{toc}{\protect\setcounter{tocdepth}{1}}
%%%%%%%%%%%%%%%%%%%%%%

% ----------------------------------------------------------------
\bibliographystyle{amsplain}
\bibliography{$HOME/Academie/Recherche/_latex/bibliotheca}%
\end{document}